\documentclass[english]{amsart}

\usepackage{amsmath}

\usepackage{amssymb}

\usepackage{amscd} \usepackage{tabularx}
\usepackage{enumerate}
\usepackage[all,cmtip,line]{xy}

\newcommand{\F}{\mathbb{F}} 

\def\RCS$#1: #2 ${\expandafter\def\csname RCS#1\endcsname{#2}}
\RCS$Revision: 191 $
\RCS$Date: 2011-06-20 19:15:39 -0500 (Mon, 20 Jun 2011) $

\newcommand{\into}{\hookrightarrow}

\newcommand{\varepsilonbar}{\overline{\varepsilon{}}}

 \newcommand{\A}{\mathbb{A}}
\newcommand{\Q}{\mathbb{Q}}

\newcommand{\C}{\mathbb{C}}

\newcommand{\rhobar}{\overline{\rho}} 
 
\newcommand{\rbar}{\bar{r}}

\newcommand{\GL}{\operatorname{GL}}

\newcommand{\PGL}{\operatorname{PGL}}

 \newcommand{\Qbar}{\overline{\Q}}

 \newcommand{\Qp}{\Q_p}
 
\newcommand{\Qpbar}{\overline{\Q}_p}
\newcommand{\Qlbar}{\overline{\Q}_{l}}

\newcommand{\Fpbar}{\overline{\F}_p}

\newcommand{\Flbar}{\overline{\F}_l} \newcommand{\Fbar}{\overline{\F}}

\newcommand{\proj}{\operatorname{proj}}

\newcommand{\PSL}{\operatorname{PSL}}

\usepackage{amsthm}

 \newtheorem{ithm}{Theorem}
\newtheorem{thm}{Theorem}[subsection]

 \newtheorem{prop}[thm]{Proposition}
 \theoremstyle{definition}
 \theoremstyle{definition}
 \theoremstyle{remark}

\numberwithin{equation}{subsection}

\theoremstyle{definition}

\usepackage{hyperref}
\begin{document}
\title[Ordinary lifts of Hilbert
  modular forms, II]  {Congruences between Hilbert modular forms:
    constructing ordinary lifts, II}

\author{Thomas Barnet-Lamb}\email{tbl@brandeis.edu}\address{Department of Mathematics, Brandeis University}
\author{Toby Gee} \email{toby.gee@imperial.ac.uk} \address{Department of
  Mathematics, Imperial College London} \author{David Geraghty}
\email{geraghty@math.princeton.edu}\address{Princeton University and
  Institute for Advanced Study} 
\begin{abstract}In this note we improve on the results of our earlier paper
  \cite{blggord}, proving a near-optimal theorem on the existence of
  ordinary lifts of a mod $l$ Hilbert modular form for any odd prime
  $l$.
\end{abstract}
\maketitle
\tableofcontents
\section{Introduction.}\label{sec:intro}Let $F$ be a totally real
field with absolute Galois group $G_F$, and let $l$ be an odd prime
number. In our earlier paper \cite{blggord}, we proved a general
result on the existence of ordinary modular lifts of a given modular
representation $\rhobar:G_F\to\GL_2(\Flbar)$; we refer the reader to
the introduction of \emph{op. cit.} for a detailed discussion of the
problem of constructing such a lift, and of our techniques for doing
so.

The purpose of this paper is to improve on the hypotheses imposed on
$\rhobar$, removing some awkward assumptions on its image; in
particular, if $l=3$ then the results of \cite{blggord} were limited
to the case that $\rhobar$ was induced from a quadratic character,
whereas our main theorem is the following. 
\begin{ithm}
  \label{thm:main thm intro version}Suppose that $l>2$ is prime, that $F$ is a
  totally real field, and that $\rhobar:G_F\to\GL_2(\Flbar)$ is
  irreducible and modular. Assume that $\rhobar|_{G_{F_v}}$ is
  reducible at all places $v|l$ of $F$, and that
  $\rhobar|_{G_{F(\zeta_l)}}$ is irreducible. If $l=5$, assume further
  that the projective image of $\rhobar$ is not isomorphic to either
  $\PGL_2(\F_5)$ or $\PSL_2(\F_5)$.

  Then $\rhobar$ has a modular lift $\rho:G_F\to\GL_2(\Qbar_l)$ which
  is ordinary at all places $v|l$.
\end{ithm}
(Note that the assumption that $\rhobar|_{G_{F_v}}$ is reducible at
all places $v|l$ of $F$ is necessary.) Our methods are based on those
of \cite{blggord}. The reason that we are now able to prove a stronger
result is that the automorphy lifting results that we employed in
\cite{blggord} have since been optimised in \cite{BLGGT} and
\cite{jack}; in particular, we make extensive use of the results of
the appendix to \cite{blgg11-serre-weights-for-U2}, which improves on
a lifting result of \cite{BLGGT}, and classifies the subgroups of
$\GL_2(\Flbar)$ which are adequate in the sense of \cite{jack}. In
Section \ref{sec:adequate} we use these results to prove Theorem
\ref{thm:main thm intro version}, except in the case that $l=3$ and
the projective image of $\rhobar(G_{F(\zeta_3)})$ is isomorphic to
$\PSL_2(\F_3)$. In this case the adequacy hypothesis we require fails,
but in Section \ref{sec:non-adequate} we handle this case by making
use of the Langlands--Tunnell theorem.

We do not know how to remove the extra hypothesis when $l=5$; since
our techniques rely on automorphy lifting theorems, it seems likely
that at least some additional assumption is needed in these cases,
cf. hypothesis (3.2.3)(3) of \cite{kis04}.

\subsection{Notation}If $M$ is a field, we let $G_M$ denote its
absolute Galois group. We write $\bar{\varepsilon}$ for the mod $l$
cyclotomic character. We fix an algebraic closure $\Qbar$ of $\Q$, and
regard all algebraic extensions of $\Q$ as subfields of $\Qbar$. For
each prime $p$ we fix an algebraic closure $\Qpbar$ of $\Qp$, and we
fix an embedding $\Qbar\into\Qpbar$. In this way, if $v$ is a finite
place of a number field $F$, we have a homomorphism $G_{F_v}\into
G_F$. We also fix an embedding $\Qbar\into\C$.

We normalise the definition of Hodge--Tate weights so that all the
Hodge--Tate weights of the $l$-adic cyclotomic character $\varepsilon$ are $-1$. We
refer to a two-dimensional potentially crystalline representation with
all pairs of labelled Hodge--Tate weights equal to $\{0,1\}$ as a
weight $0$ representation.

If $F$ is a totally real field, then a continuous representation
$\rbar : G_{F} \to \GL_2(\Flbar)$ is said to be \emph{modular} if
there exists a regular algebraic automorphic representation $\pi$ of
$\GL_2(\A_{F})$ such that $\rbar_{l}(\pi)\cong \rbar$, where
$r_l(\pi)$ is the $l$-adic Galois representation associated to $\pi$.

\section{The adequate case}\label{sec:adequate}\subsection{} The notion of an \emph{adequate} subgroup of $\GL_n(\Flbar)$ is
defined in \cite{jack}. We will not need to make use of the actual
definition; instead, we will use the following classification
result. Note that by definition an adequate subgroup of
$\GL_n(\Flbar)$ necessarily acts irreducibly on $\Flbar^n$.

\begin{prop} 
  \label{prop:adequacy for n=2} Suppose that $l>2$ is a prime, and
  that $G$  is a finite subgroup of $\GL_2(\Fbar_l)$ which acts
  irreducibly on $\Fbar_l^2$.  Then precisely one of the following is
  true:
\begin{itemize}
\item We have $l=3$, and the image of $G$ in $\PGL_2(\Fbar_3)$ is  conjugate to $\PSL_2(\F_3)$.
\item We have $l=5$, and the image of $G$ in $\PGL_2(\Fbar_5)$ is  conjugate to
 $\PGL_2(\F_5)$ or $\PSL_2(\F_5)$.
\item $G$ is adequate.
\end{itemize}
\end{prop}
\begin{proof}
  This is Proposition A.2.1 of \cite{blgg11-serre-weights-for-U2}.
\end{proof}
In the case that $\rhobar(G_{F(\zeta_l)})$ is adequate, our main
result follows exactly as in section 6 of \cite{blggord}, using the
results of Appendix A of \cite{blgg11-serre-weights-for-U2} (which in
turn build on the results of \cite{BLGGT}). We obtain the following
theorem.
\begin{thm}
  \label{main thm in adequate case}Suppose that $l>2$ is prime, that
  $F$ is a totally real field, and that $\rhobar:G_F\to\GL_2(\Flbar)$
  is irreducible and modular. Suppose also that
  $\rhobar(G_{F(\zeta_l)})$ is adequate. Then:
  \begin{enumerate}
  \item There is a finite solvable extension of totally real fields
    $L/F$ which is linearly disjoint from $\overline{F}^{\ker\rhobar}$
    over $F$, such that $\rhobar|_{G_L}$ has a modular lift
    $\rho_L:G_L\to\GL_2(\Qlbar)$ of weight $0$ which is ordinary at all places $v|l$.
  \item If furthermore $\rhobar|_{G_{F_v}}$ is reducible at all places
    $v|l$, then $\rhobar$ itself has a modular lift
    $\rho:G_F\to\GL_2(\Qlbar)$ of weight $0$ which is ordinary at all places $v|l$.
  \end{enumerate}

\end{thm}
\begin{proof}Firstly, note that (2) is easily deduced from (1) using
  the results of Section 3 of \cite{gee061} (which build on Kisin's
  reinterpretation of the Khare--Wintenberger method). Indeed, the
  proofs of Theorems 6.1.5 and 6.1.7 of \cite{blggord} go through
  unchanged in this case.

  Similarly, (1) is easily proved in the same way as Proposition 6.1.3
  of \cite{blggord} (and in fact the proof is much shorter). Firstly,
  note that the proof of Lemma 6.1.1 of \cite{blggord} goes through
  unchanged to show that there is a finite solvable extension of
  totally real fields $L/F$ which is linearly disjoint from
  $\overline{F}^{\ker\rhobar}$ over $F$, such that $\rhobar|_{G_L}$
  has a modular lift $\rho':G_L\to\GL_2(\Qlbar)$ of weight $0$ which
  is potentially crystalline at all places dividing $l$, and in addition both
  $\rhobar|_{G_{L_w}}$ and $\varepsilonbar|_{G_{L_w}}$ are trivial for
  each place $w|l$ (and in particular, $\rhobar|_{G_{L_w}}$ admits an
  ordinary lift of weight $0$), and $\rhobar$ is unramified at all
  finite places. By Lemma 4.4.1 of \cite{geekisin}, $\rho'|_{G_{L_w}}$ is
  potentially diagonalizable in the sense of \cite{BLGGT} for all
  places $w|l$ of $L$.

  Choose a CM quadratic extension $M/L$ which is linearly disjoint
  from $L(\zeta_l)$ over $L$, in which all places of $L$ dividing $l$
  split. We can now apply Theorem A.4.1 of
  \cite{blgg11-serre-weights-for-U2} (with $F'=F=M$, $S$ the set of
  places of $L$ dividing $l$, and $\rho_v$ an ordinary lift of
  $\rhobar|_{G_{L_w}}$ for each $w|l$) to see that $\rhobar|_{G_M}$
  has an ordinary automorphic lift $\rho_M:G_M\to\GL_2(\Qlbar)$ of
  weight $0$.

The argument of the last paragraph of the proof of Proposition 6.1.3
of \cite{blggord} (which uses the Khare--Wintenberger method to
compare deformation rings for $\rhobar|_{G_L}$ and $\rhobar|_{G_M}$)
now goes over unchanged to complete the proof.  \end{proof}

\section{An inadequate
  case}\label{sec:non-adequate}\subsection{}We now consider the case
that $l=3$ and $\rhobar|_{G_{F(\zeta_3)}}$ is irreducible, but
$\rhobar(G_{F(\zeta_3)})$ is not adequate. By Proposition
\ref{prop:adequacy for n=2}, this means that the projective image of
$\rhobar(G_{F(\zeta_3)})$ is isomorphic to $\PSL_2(\F_3)$, and is in
particular solvable. We now use the Langlands--Tunnell theorem to
prove our main theorem in this case.

\begin{thm}
  \label{thm:main thm in PSL2F3 case}Suppose that $F$ is a totally
  real field, and that $\rhobar:G_F\to\GL_2(\Fbar_3)$ is irreducible
  and modular. Assume that $\rhobar|_{G_{F_v}}$ is reducible at all
  places $v|3$ of $F$, and that the projective image of
  $\rhobar(G_{F(\zeta_3)})$ is isomorphic to $\PSL_2(\F_3)$. 

  Then $\rhobar$ has a modular lift $\rho:G_F\to\GL_2(\Qbar_3)$ which
  is ordinary at all places $v|3$.
\end{thm}
\begin{proof}Firstly, note that since the projective image of
  $\rhobar(G_{F(\zeta_3)})$ is isomorphic to $\PSL_2(\F_3)$, the
  projective image of $\rhobar$ itself is isomorphic to $\PSL_2(\F_3)$
  or $\PGL_2(\F_3)$ (see, for example, Theorem 2.47(b) of
  \cite{MR1605752}).

 Choose a finite solvable extension of totally real fields $L/F$
 which is linearly disjoint from  $\overline{F}^{\ker\rhobar}$ over
 $F$, with the further property that $\rhobar|_{G_{L_w}}$ is
 unramified for each place $w|l$ of $L$. Exactly as in the proof of
 Theorem \ref{main thm in adequate case}, by the results of  Section 3
 of \cite{gee061} it suffices to show that $\rhobar|_{G_L}$ has a
 modular lift of weight $0$ which is potentially crystalline at each
 place $w|l$. By Hida theory, it in fact suffices to find some ordinary
 modular lift of $\rhobar|_{G_L}$ (not necessarily of weight $0$).

 Since the projective image of $\rhobar$ is isomorphic to
 $\PSL_2(\F_3)$ or $\PGL_2(\F_3)$, the image of $\rhobar$ is contained
 in $\Fpbar^\times\GL_2(\F_3)$. Then the Langlands--Tunnell theorem
 implies that $\rhobar|_{G_L}$ has a modular lift $\rho$ corresponding to a Hilbert
 modular form of parallel weight one. This follows from the discussion after Theorem
 5.1 of \cite{wiles-fermat} which also shows that the natural map
 $\rho(G_L)\to\rhobar(G_L)$ may be assumed to be an isomorphism. Since
 $\rhobar|_{G_{L_w}}$ is unramified at each place $w|l$ of $L$, this
 implies that $\rho$ is ordinary, as required. \end{proof} Finally, we deduce
  our main result from Theorems \ref{main thm in adequate case} and
  \ref{thm:main thm in PSL2F3 case}.
\begin{thm}
  \label{thm:main thm}Suppose that $l>2$ is prime, that $F$ is a
  totally real field, and that $\rhobar:G_F\to\GL_2(\Flbar)$ is
  irreducible and modular. Assume that $\rhobar|_{G_{F_v}}$ is
  reducible at all places $v|l$ of $F$, and that
  $\rhobar|_{G_{F(\zeta_l)}}$ is irreducible. If $l=5$, assume further
  that the projective image of $\rhobar$ is not isomorphic to either
  $\PGL_2(\F_5)$ or $\PSL_2(\F_5)$.

  Then $\rhobar$ has a modular lift $\rho:G_F\to\GL_2(\Qbar_l)$ which
  is ordinary at all places $v|l$.
\end{thm}
\begin{proof}
  If $l=3$ and the projective image of $\rhobar(G_{F(\zeta_3)})$ is
  isomorphic to $\PSL_2(\F_3)$, then the result follows from Theorem
  \ref{thm:main thm in PSL2F3 case}. In all other cases we see from
  Proposition \ref{prop:adequacy for n=2} that
  $\rhobar(G_{F(\zeta_l)})$ is adequate (note that by (for example)
  Theorem 2.47(b) of \cite{MR1605752}, $\proj\rhobar(G_{F(\zeta_l)})$
  is isomorphic to $\PGL_2(\F_5)$ or $\PSL_2(\F_5)$ if and only if
  $\proj\rhobar(G_F)$ is isomorphic to $\PGL_2(\F_5)$ or
  $\PSL_2(\F_5)$) 
  and the result follows from Theorem \ref{main thm in adequate
    case}(2).
\end{proof}

\bibliographystyle{amsalpha} 
\bibliography{barnetlambgeegeraghty}
\end{document}